\documentclass[a4paper, 12pt,oneside,reqno]{amsart}
\usepackage[a4paper]{geometry}
\usepackage{amssymb,amsmath,amsthm}
\usepackage[T2A]{fontenc}
\usepackage[breaklinks=true]{hyperref}

\usepackage{xcolor}

\usepackage{bbm}
\usepackage{graphicx}

\def\Var{\mathrm{Var}}
\def\Assos{\mathrm{Assos}}
\def\Alt{\mathrm{Alt}}
\def\dim{\mathrm{dim}}
\def\Der{\mathrm{Der}}
\def\Nov{\mathrm{Nov}}
\def\Com{\mathrm{Com}}
\def\As{\mathrm{As}}
\def\Bi{\mathrm{Bi}}

\def\Zinb{\mathrm{Zinb}}

\def\wt{\mathrm{wt}}

\def\Der {\mathop {\fam 0 Der}\nolimits}

\let\<\langle
\let\>\rangle

\newtheorem{definition}{Definition}
\newtheorem{lemma}{Lemma}
\newtheorem{proposition}{Proposition}
\newtheorem{theorem}{Theorem}
\newtheorem{corollary}{Corollary}

\title{White Manin product and Hadamard product}

\author{P. S. Kolesnikov}

\address{Sobolev Institute of Mathematics, Novosibirsk (Russia)}

\email{pavelsk77@gmail.com}

\author{B. K. Sartayev$^{*}$}

\address{Narxoz University, Almaty, Kazakhstan}

\email{baurjai@gmail.com}


\thanks{${}^{*}$Corresponding author: Bauyrzhan Sartayev   (baurjai@gmail.com)}

\begin{document}
    
\maketitle

\begin{abstract}
In this paper, we consider three types of operads: alternative, assosymmetric, and bicommutative. We prove that the Hadamard product of these operads with the Novikov operad coincides with their white Manin product. As an application, we identify a variety of algebras in which all algebras are special.
\end{abstract}

\section{Introduction}

The weight criterion for commutative-associative algebras with derivation plays an important role in constructing a basis of the free Novikov algebra \cite{DzhLofwall} and in proving that every algebra in the variety of Novikov algebras is special \cite{BCZ2017}. That is, every Novikov algebra can be embedded into appropriate associative-commutative algebra with a derivation $d$.

Let us recall the crucial definition of the weight function:
\begin{definition}
Let $\wt$ be a weight function defined on monomials of $\Var\<X^{(d)}\>$ as follows:
\[
\wt:\Var\<X^{(d)}\>\rightarrow \mathbb{Z},
\]
\[
\wt(x^{(j)}) =j-1,\quad x\in X,\ j\ge 0,
\]
\[
\wt(uv) = \wt(u)+\wt(v)
\]
for all monomials $u,v\in \Var\<X^{(d)}\>$, where $\Var^{(d)}$denotes the variety of algebras $\Var$ equipped with a derivation $d$, and $\Var\<X^{(d)}\>$ is the corresponding free algebra generated by a countable set $X$.
\end{definition}
If we define a new operation $\prec$ in a commutative-associative algebra with a derivation $d$ as 
\[
a\prec b=ad(b),
\]
then the resulting algebra with the operation $\prec$ forms a Novikov algebra. Moreover, every monomial in $\Com\<X^{(d)}\>$ of the weight $-1$ can be expressed using the operation $\prec$. As stated above, under the functor 
\[
\tau(a\prec b)=ad(b)
\]
every Novikov algebra is special. More generally, for an arbitrary variety $\Var$, we can define new operations $\prec$ and $\succ$ on $\Var\<X^{(d)}\>$
as
\[
a\succ b=d(a)b\;\textrm{and}\;a\prec b=ad(b),
\]
and define a functor $\tau$ relative to these operations.

Novikov algebras play an important role in the theory of operads. We denote by $\Nov$ an operad derived from the variety of Novikov algebras.
Since there is a one-to-one correspondence between a variety of algebras $\Var$ and the quadratic operad derived from it, we will use the same terminology for both throughout this paper.
In \cite{KSO2019}, it was proved that for an arbitrary quadratic operad $\Var$, the free algebra $\Der\Var\<X\>$ obtained from the white Manin product of operads $\Var\circ\Nov:=\Der\Var$ can be embedded into $\Var\<X^{(d)}\>$ under the functor $\tau$. However, not every algebra of $\Der\Var$ can be embedded into some algebra of $\Var^{(d)}$, see \cite{KolMashSar}.

The condition that a monomial in $\Var\<X^{(d)}\>$ has weight $-1$ is necessary for expressing it in terms of $\succ$ and $\prec$. For example, every monomial of weight $-1$ in $\As\<X^{(d)}\>$ can be expressed by operations $\succ$ and $\prec$ \cite{erlagol2021}, where $\As$ is a variety of associative algebras. As an application, \cite{DauSar} provides a construction of a basis for the algebra $\Der\As\<X\>$ and establishes that
\[
\dim(\Der\As(n))=n!\binom{2n-2}{n-1}.
\]
However, it is not always sufficient. For example, in the variety of Zinbiel algebras, not every monomial of weight $-1$ in $\Zinb\<X^{(d)}\>$ belongs to $\Der\Zinb\<X\>$. We say that the variety of algebras $\Var$ satisfies the weight criterion if every monomial of weight $-1$ in $\Var\<X^{(d)}\>$ can be expressed using the operations $\succ$ and $\prec$.

If $\Var$ satisfies the weight criterion, then it has the following important properties:

\begin{lemma}\label{lem:Weight-criterion}
\cite{KolMashSar} Let $\Var $ be a binary operad such that $\Var\circ \Nov = \Var\otimes \Nov $.
Then for every set $X$ an element $f\in \Var\langle X^{(d)}\rangle $
belongs to $\Der\Var \langle X\rangle $ if and only if
$\wt (f)=-1$.
\end{lemma}

\begin{theorem}\label{thm:Weight-Embedding}
\cite{KolMashSar} If
$\Var $ is a binary operad such that
$\Var\circ \Nov = \Var\otimes \Nov $
then every $\Der\Var $-algebra is special.
\end{theorem}

Let us list some works on alternative, assosymmetric, bicommutative, and Novikov algebras, as well as their associated operads. The dimensions of the assosymmetric operad and some examples are provided in \cite{assos1, assos2}. Results on the embedding of metabelian Novikov algebras into bicommutative algebras, as well as the invariant properties of bicommutative algebras, can be found in \cite{bicom2, bicom3, bicom4}. It is well known that alternative algebras, when equipped with the commutator, form Malcev algebras. However, a major open problem is whether every Malcev algebra can be embedded into an appropriate alternative algebra. Only partial results are currently known \cite{Meta-Mal, Binaryperm}. A recent study on Novikov algebras and operads can be found in \cite{MetNov}.

In this paper, we prove that the varieties of alternative, assosymmetric, and bicommutative algebras satisfy the weight criterion. In the final section, we present an application of our results.

\section{Bicommutative algebra}

An algebra is called bicommutative if it satisfies the following identities:
\begin{equation}\label{lcom}
a(bc)=b(ac)
\end{equation}
and
\begin{equation}\label{rcom}
(ab)c=(ac)b.
\end{equation}

\begin{theorem}\label{bicom-1}
Every monomial of weight $-1$ in $\Bi\Com\<X^{(d)}\>$ can be expressed by the operations $\prec$ and $\succ$.
\end{theorem}
\begin{proof}
Let us prove this by induction on the length of monomials. For $n=2$, the result is obvious. For $n=3$, it suffices to prove the result for the monomials
\[
x_1(x_2'x_3'),\;x_1'(x_2x_3'),\;x_1'(x_2'x_3),\;x_1''(x_2x_3),\;x_1(x_2''x_3),\;x_1(x_2x_3'')
\]
due to the symmetry of identities (\ref{lcom}) and (\ref{rcom}). By (\ref{lcom}), we have
\[
x_1(x_2'x_3')=x_2'(x_1x_3')=x_2\succ(x_1\prec x_3),\; x_1'(x_2x_3')=x_1\succ(x_2\prec x_3),\; x_1'(x_2'x_3)=x_2\succ(x_1\succ x_3),
\]
\[
x_1''(x_2x_3)=x_2(x_1''x_3)=x_2\prec(x_1\succ x_3)-x_1\succ(x_2\prec x_3),
\]
\[
x_1(x_2''x_3)=x_1\prec(x_2\succ x_3)-x_2\succ(x_1\prec x_3)
\]
and
\[
x_1(x_2x_3'')=x_1\prec(x_2\prec x_3)-x_2\succ(x_1\prec x_3).
\]
Suppose the result holds for all monomials of length at most $n-1$. To prove the result for an arbitrary monomial of length $n$, we first recall the basis of the algebra $\Bi\Com\<X\>$ (see \cite{Ism}). This basis consists of monomials of the following form:
\begin{equation}\label{bicom-monomial}
x_{i_1}(\cdots (x_{i_{n-1}}((\cdots((x_{i_n} x_{j_1})x_{j_2})\cdots)x_{j_m}))\cdots),
\end{equation}
where $n$ and $m$ are any positive integers, $i_n\geq i_{n-1}\geq\ldots\geq i_1$ and $j_1\geq j_2\geq\ldots\geq j_m$.
It is worth noting that the generators $x_{i_1},x_{i_2},\ldots,x_{i_n}$ and $x_{j_1},x_{j_2},\ldots,x_{j_m}$ are rearrangeable in (\ref{bicom-monomial}), i.e.,

\begin{multline*}
x_{i_1}(\cdots (x_{i_{n-1}}((\cdots((x_{i_n} x_{j_1})x_{j_2})\cdots)x_{j_m}))\cdots)=\\
x_{\sigma{(i_1)}}(\cdots (x_{\sigma{(i_{n-1)}}}((\cdots((x_{\sigma{(i_m)}} x_{j_1})x_{j_2})\cdots)x_{j_m}))\cdots)=\\    
x_{i_1}(\cdots (x_{i_{n-1}}((\cdots((x_{i_n} x_{\rho(j_1)})x_{\rho(j_2)})\cdots)x_{\rho(j_m)}))\cdots),
\end{multline*}
where $\sigma\in S_{(i_1,\ldots,i_n)}$ and $\rho\in S_{(j_1,\ldots,j_m)}$.

So, the basis of $\Bi\Com\<X^{(d)}\>$ is the same set of monomials with generators
$$X^{(d)}=X\cup X'\cup X''\cdots\cup X^{(n)}\cup\cdots.$$
Let us consider $2$ cases.

Case 1: In (\ref{bicom-monomial}), suppose there are at least two generators with nonzero derivations. By applying identities (\ref{lcom}) and (\ref{rcom}), we always can rewrite (\ref{bicom-monomial}) in the form
\[
x_{l_1}^{(l_1')}(\cdots (x_{l_{p-1}}^{(l_{p-1}')}((\cdots(((x_{i_1}(\cdots (x_{i_{r-1}}((\cdots((x_{i_r}^{(t)} x_{j_1})x_{j_2})\cdots)x_{j_m}))\cdots))x_{k_1}^{(k_1')})x_{k_2}^{(k_2')})\cdots)x_{k_q}^{(k_q')}))\cdots),
\]
or equivalently,
\[
x_{l_1}^{(l_1')}(\cdots (x_{l_{p-1}}^{(l_{p-1}')}((\cdots(((x_{i_1}(\cdots (x_{i_{r-1}}((\cdots((x_{i_r} x_{j_1}^{(t)})x_{j_2})\cdots)x_{j_m}))\cdots))x_{k_1}^{(k_1')})x_{k_2}^{(k_2')})\cdots)x_{k_q}^{(k_q')}))\cdots).
\]
That is, all generators with derivations, except the chosen one, are placed outside the bracketed expression 
\[
x_{i_1}(\cdots (x_{i_{r-1}}((\cdots((x_{i_r} ^{(t)}x_{j_1})x_{j_2})\cdots)x_{j_m}))\cdots)
\]
or
\[
x_{i_1}(\cdots (x_{i_{r-1}}((\cdots((x_{i_r} x_{j_1}^{(t)})x_{j_2})\cdots)x_{j_m}))\cdots),
\]
respectively. By the inductive hypothesis, these monomials can be expressed using the operations $\succ$ and $\prec$.

Case 2: In (\ref{bicom-monomial}), suppose the generator $x_k^{(n)}$ appears. Without loss of generality, we can rewrite (\ref{bicom-monomial}) as
\begin{multline*}
x_{i_1}(\cdots (x_{i_{r-1}}((\cdots((x_{i_r} ^{(n)}x_{i_{r+1}})x_{i_{r+2}})\cdots)x_{i_{n+1}}))\cdots)=\\
x_{i_1}\prec(\cdots (x_{i_{r-1}}\prec((\cdots((x_{i_r}\succ x_{i_{r+1}})\succ x_{i_{r+2}})\cdots)\succ x_{i_{n+1}}))\cdots)\\
-\sum_{i_2',\ldots,i_{n+1}'<n} x_{i_1}(\cdots (x_{i_{r-1}}^{(i_{r-1}')}((\cdots((x_{i_r} ^{(i_{r}')}x_{i_{r+1}}^{(i_{r+1}')})x_{i_{r+2}}^{(i_{r+2}')})\cdots)x_{i_{n+1}}^{(i_{n+1}')}))\cdots).
\end{multline*}
All monomials in the summation correspond to case 1.
\end{proof}

\begin{corollary}
From Lemma \ref{lem:Weight-criterion} and Theorem \ref{bicom-1}, we obtain
\[
\Bi\Com\circ\Nov=\Bi\Com\otimes\Nov
\]
and
\[
\dim(\Der\Bi\Com(n))=(2^n-2)\binom{2n-2}{n-1}.
\]
\end{corollary}

\section{Alternative algebra}
An algebra is called an alternative if it satisfies the following identities:
\begin{equation}\label{lasym}
(ab)c-a(bc)=-(ba)c+b(ac)
\end{equation}
and
\begin{equation}\label{rasym}
(ab)c-a(bc)=-(ac)b+a(cb).
\end{equation}

\begin{theorem}\label{alt-1}
Every monomial of weight $-1$ in $\Alt\<X^{(d)}\>$ can be expressed by the operations $\prec$ and $\succ$.
\end{theorem}
\begin{proof}
Let us prove this by induction on the length of monomials. For $n=2$, the result is obvious. For $n=3$, it suffices to prove the result for the monomials
\[
x_1(x_2'x_3'),\;x_1'(x_2x_3'),\;x_1'(x_2'x_3),\;x_1''(x_2x_3),\;x_1(x_2''x_3),\;x_1(x_2x_3''),
\]
and for right normed monomials, the calculations are the same. By (\ref{lasym}) and (\ref{rasym}), we have
\[
x_1(x_2'x_3')=(x_1x_2')x_3'+(x_2'x_1)x_3'-x_2'(x_1x_3')=(x_1\prec x_2)\prec x_3+(x_2\succ x_1)\prec x_3-x_2\succ(x_1\prec x_3),
\]
\[
x_1'(x_2x_3')=x_1\succ(x_2\prec x_3),\; x_1'(x_2'x_3)=x_2\succ(x_1\succ x_3),
\]
\begin{multline*}
x_1''(x_2x_3)=(x_1''x_2)x_3+(x_2x_1'')x_3-x_2(x_1''x_3)=\\
(x_1\succ x_2)\succ x_3-x_1\succ(x_2\succ x_3)+(x_1\succ x_3)\prec x_2-x_1\succ(x_3\prec x_2)\\
+(x_2\prec x_1)\succ x_3-x_2\succ(x_1\succ x_3)+(x_2\succ x_3)\prec x_1-x_2\succ(x_3\prec x_1)\\
-x_2\prec(x_1\succ x_3)+(x_2\prec x_1)\prec x_3+(x_1\succ x_2)\prec x_3-x_1\succ(x_2\prec x_3)
\end{multline*}
\[
x_1(x_2''x_3)=x_1\prec(x_2\succ x_3)-(x_1\prec x_2)\prec x_3-(x_2\succ x_1)\prec x_3+x_2\succ(x_1\prec x_3),
\]
and
\[
x_1(x_2x_3'')=x_1\prec(x_2\prec x_3)-(x_1\prec x_2)\prec x_3-(x_2\succ x_1)\prec x_3+x_2\succ(x_1\prec x_3).
\]
Suppose that the result holds for all monomials of length at most $n-1$. To prove the result for an arbitrary monomial of length $n$, we consider $\mathcal{M}_n=(\mathcal{M}_k)(\mathcal{M}_{n-k})$, where $\mathcal{M}_i$ is an arbitrary monomial of length $i$. Without loss of generality, we may assume that $\wt(\mathcal{M}_k)\leq -1$. If $\wt(\mathcal{M}_k)=-1$, then the result is obvious by the inductive hypothesis. So, we may assume that $\wt(\mathcal{M}_k)<-1$. To prove the main result, we first prove three supporting lemmas.
\begin{lemma}\label{lemma1}
If $\mathcal{M}_n=(\mathcal{M}_k)(\mathcal{M}_{n-k})$, then $\mathcal{M}_n$ can be rewritten as a sum of monomials, each containing at least one generator with a derivation in every term of the decomposition.
\end{lemma}
\begin{proof}
If both $\mathcal{M}_{k}$ and $\mathcal{M}_{n-k}$ each contain at least one generator with a derivation, then the result follows. Suppose that this is not the case, that is, neither $\mathcal{M}_k$ nor $\mathcal{M}_{n-k}$ contains a generator with a derivation. Then
\begin{multline*}
(\mathcal{M}_k)(\mathcal{M}_{n-k})=(\mathcal{M}_k)((\mathcal{M}_{p})(\mathcal{M}_{n-k-p}))=^{(\ref{lasym})}\\
((\mathcal{M}_k)(\mathcal{M}_{p}))(\mathcal{M}_{n-k-p})+((\mathcal{M}_{p})(\mathcal{M}_k))(\mathcal{M}_{n-k-p})-(\mathcal{M}_{p})((\mathcal{M}_k)(\mathcal{M}_{n-k-p})),
\end{multline*}
i.e., if $\mathcal{M}_k$ does not contain a generator with a derivation, and both $\mathcal{M}_{p}$ and $\mathcal{M}_{n-k-p}$ contain at least one generator with a derivation, then using {(\ref{lasym})}, we can rewrite $(\mathcal{M}_k)(\mathcal{M}_{n-k})$ as a sum of monomials, each containing at least one generator with a derivation in every term of the decomposition. Analogically, we can use (\ref{rasym}) to repeat the same trick if $\mathcal{M}_{n-k}$ does not contain a generator with derivation:
\begin{multline*}
(\mathcal{M}_k)(\mathcal{M}_{n-k})=((\mathcal{M}_p)(\mathcal{M}_{k-p}))(\mathcal{M}_{n-k})=^{(\ref{rasym})}\\
(\mathcal{M}_p)((\mathcal{M}_{k-p})(\mathcal{M}_{n-k}))-((\mathcal{M}_p)(\mathcal{M}_{n-k}))(\mathcal{M}_{k-p})+(\mathcal{M}_p)((\mathcal{M}_{n-k})(\mathcal{M}_{k-p})).
\end{multline*}
By repeating this process several times, we can finally rewrite $\mathcal{M}_n$ as a sum of monomials, each containing at least one generator with a derivation in every term of the decomposition.
\end{proof}

\begin{lemma}\label{lemma2}
Any monomial in $\Alt\<X\>$ can be rewritten as a sum of monomials of the form 
\[
O_{x_{i_1}}O_{x_{i_2}}\cdots O_{x_{i_{n-1}}}x_{i_n},
\]
where $O_x$ denotes either a left or right multiplication operator.
\end{lemma}
\begin{proof}
We prove it by induction on the length of monomials. The result is obvious for monomials of length $2$ and $3$.

For an arbitrary monomial $\mathcal{M}_n=(\mathcal{M}_k)(\mathcal{M}_{n-k})$, by inductive hypothesis, we may assume that both $\mathcal{M}_k$ and $\mathcal{M}_{n-k}$ satisfy the condition of the lemma. Furthermore, for $\mathcal{M}_{n-k}=(\mathcal{M}_{p})(\mathcal{M}_{n-k-p})$, we may assume that $\mathcal{M}_{n-k-p}=x_j$. Applying $(\ref{lasym})$, we obtain
\[
(\mathcal{M}_k)(\mathcal{M}_{n-k})=(\mathcal{M}_k)((\mathcal{M}_{p})x_j)=
((\mathcal{M}_k)(\mathcal{M}_{p}))x_j+((\mathcal{M}_{p})(\mathcal{M}_k))x_j-(\mathcal{M}_{p})((\mathcal{M}_k)x_j).
\]
By inductive hypothesis, the monomials $((\mathcal{M}_k)(\mathcal{M}_{p}))x_j$ and $((\mathcal{M}_{p})(\mathcal{M}_k))x_j$ satisfy the condition of the lemma. Analogically, by repeatedly applying (\ref{lasym}) and (\ref{rasym}) to $(\mathcal{M}_{p})((\mathcal{M}_k)x_j)$, we reduce the length of $\mathcal{M}_{p}$ and finally obtain the result.
\end{proof}

By Lemma \ref{lemma1}, we may assume that if $\mathcal{M}_n=(\mathcal{M}_k)(\mathcal{M}_{n-k})$ then both $\mathcal{M}_k$ and $\mathcal{M}_{n-k}$ contain at least one generator with a derivation.

\begin{lemma}\label{lemma3}
The monomial
\[
O_{x_{i_1}}O_{x_{i_2}}\cdots O_{x_{i_{n-1}}}x_{i_n},
\]
in $\Alt\<X\>$ can be written as a sum of monomials of the form
\[
O_{x_{j_1}}O_{x_{j_2}}\cdots O_{x_{j_{n-2}}} O_{x_{j_{n-1}}}x_{i_1}.
\]
\end{lemma}
\begin{proof}
We prove it by induction on the length of monomials. The result is obvious for monomials of length $2$. For $R_{x_{i_1}}O_{x_{i_2}}x_{i_3}$ and $L_{x_{i_1}}O_{x_{i_2}}x_{i_3}$, it suffices to apply (\ref{rasym}) and (\ref{lasym}), respectively. Similarly, for monomials
\[
R_{x_{i_1}}O_{x_{i_2}}\cdots O_{x_{i_{n-1}}}x_{i_n}\;\;
\textit{and}\;\;
L_{x_{i_1}}O_{x_{i_2}}\cdots O_{x_{i_{n-1}}}x_{i_n},
\]
we apply (\ref{rasym}) and (\ref{lasym}), respectively, along with the inductive hypothesis.
\end{proof}

Now, we are ready to prove the result. Let us consider 2 cases.

Case 1: In $\mathcal{M}_n$, suppose that there are at least two generators with nonzero derivations. 
If $\mathcal{M}_n=(\mathcal{M}_k)(\mathcal{M}_{n-k})$, then by Lemma \ref{lemma1}, we can assume that both $\mathcal{M}_{k}$ and $\mathcal{M}_{n-k}$ each contain at least one generator with a derivation. Without loss of generality, we may assume that $\wt(\mathcal{M}_k)<-1$. As above, for $\mathcal{M}_k=(\mathcal{M}_{k_1})(\mathcal{M}_{k-k_1})$ both $\mathcal{M}_{k_1}$ and $\mathcal{M}_{k-k_1}$ each contain at least one generator with a derivation, and $\wt(\mathcal{M}_{k_1})<-1$. We repeat this process up to a monomial $\mathcal{M}_{k_l}$ such that $\wt(\mathcal{M}_{k_1})<-1$ and it has only one generator with derivation. By Lemma \ref{lemma1} and Lemma \ref{lemma2}, such $\mathcal{M}_{k_1}$ can be rewritten as a sum of monomials
\[
O_{x_{j_1}}\cdots O_{x_{j_{r-1}}}x^{(r)}_{j_r},
\]
that can be expressed by operations $\succ$ and $\prec$ by inductive hypothesis.

Case 2: In $\mathcal{M}_n$, suppose that the generator $x_k^{(n)}$ appears. By Lemma \ref{lemma2}, we can express $\mathcal{M}_n$ as a sum of monomials
\[
O_{x_{i_1}}\cdots O_{x^{(n)}_{i_t}}\cdots O_{x_{i_{n-1}}}x_{i_n},
\]
and by Lemma \ref{lemma3}, we can express it as a sum of monomials
\[
O_{x_{j_1}}\cdots O_{x_{j_{n-1}}}x^{(n)}_{i_t}.
\]
If $O_{x_p}y=R_{x_p}y$, then we rewrite it as $y\succ x_p$. If $O_{x_p}y=L_{x_p}y$, then we rewrite it as $x_p\prec y$. Finally, we the monomial $O_{x_{j_1}}\cdots O_{x_{j_{n-1}}}x^{(n)}_{i_t}
$ in terms of $\succ$ and $\prec$ and remained monomials correspond to Case 1.

\end{proof}

\begin{corollary}
From Lemma \ref{lem:Weight-criterion} and Theorem \ref{alt-1}, we obtain
\[
\Alt\circ\Nov=\Bi\Com\otimes\Nov.
\]
\end{corollary}

\begin{corollary}\label{assos-1}
The analog of Theorem \ref{alt-1} holds for $\Assos\<X^{(d)}\>$
and
\[
\dim(\Der\Assos(n))=(n!+2^n-\binom{n+1}{2}-1)\binom{2n-2}{n-1}.
\]
\end{corollary}
\begin{proof}
The proof of result for $\Assos\<X^{(d)}\>$ is analogical as for $\Alt\<X^{(d)}\>$.
\end{proof}

\section{Application}

The explicit calculation technique of $\Var_1\circ\Var_2$ is given in \cite{KSO2019}. Using this technique, we obtain the defining identities of operads $\Der\Alt$, $\Der\Assos$ and $\Der\Bi\Com$.

\begin{proposition}
The white Manin product of the operads $\Alt$ and $\Nov$ gives an operad $\Der\Alt$ which is defined by the following identities:
\[
(a\succ b)\prec c-a\succ(b\prec c)=-(c\succ b)\prec a+c\succ(b\prec a)
\]
and
\begin{multline*}
(a\prec b)\succ c-a\succ(b\succ c)+(a\succ c)\prec b-a\succ(c\prec b)-a\prec(b\succ c)+(a\prec b)\prec c\\
+(b\succ a)\prec c-b\succ(a\prec c)=
-(c\prec b)\succ a+c\succ(b\succ a)-(c\succ a)\prec b+c\succ(a\prec b)\\
+c\prec(b\succ a)-(c\prec b)\prec a-(b\succ c)\prec a+b\succ(c\prec a).
\end{multline*}
\end{proposition}

\begin{proposition}
The white Manin product of the operads $\Assos$ and $\Nov$ gives an operad $\Der\Assos$ which is defined by the following identities:
\[
(a\succ c)\prec b-a\succ(c\prec b)=(b\succ c)\prec a-b\succ(c\prec a)
\]
and
\begin{multline*}
(a\prec c)\succ b-a\succ(c\succ b)-(a\succ b)\prec c+a\succ(b\prec c)-a\prec(c\succ b)+(a\prec c)\prec b\\
-(c\succ a)\prec b+c\succ(a\prec b)=
(b\prec c)\succ a-b\succ(c\succ a)-(b\succ a)\prec c+b\succ(a\prec c)\\
-b\prec(c\succ a)+(b\prec c)\prec a-(c\succ b)\prec a+c\succ(b\prec a).
\end{multline*}
\end{proposition}

\begin{proposition}
The white Manin product of the operads $\Bi\Com$ and $\Nov$ gives an operad $\Der\Bi\Com$ which is defined by the following identities:
\[
(a\prec b)\prec c=(a\prec c)\prec b,
\]
\[
a\succ(b\succ c)=b\succ(a\succ c),
\]
\[
(a\succ b)\succ c-(a\succ c)\prec b=(a\succ c)\succ b-(a\succ b)\prec c
\]
and
\[
a\prec(b\prec c)-b\succ(a\prec c)=b\prec(a\prec c)-a\succ(b\prec c).
\]
\end{proposition}

\begin{theorem}
Any algebra of the variety of algebras $\Der\Var$ can be embedded into the appropriate algebra of the variety of algebras $\Var^d$, where $\Var=\Alt,\Assos,\Bi\Com$. 
\end{theorem}
\begin{proof}
Before, we proved that the weight criterion works for operads $\Alt$, $\Assos$ and $\Bi\Com$. By Lemma \ref{lem:Weight-criterion} and Theorem \ref{thm:Weight-Embedding}, we obtain the result.
\end{proof}

\subsection*{Acknowledgments}
The first author was supported by the Program of Fundamental Research RAS (project FWNF-2022-0002).


\begin{thebibliography}{99}


\bibitem{BCZ2017}
L. A. Bokut, Y. Chen, Z. Zhang, 
Gr\"obner--Shirshov bases method for Gelfand--Dorfman--Novikov algebras, 
Journal of Algebra its Applications,  16(1), 1750001, 22 pp.  (2017).


\bibitem{MetNov}
A. Dauletiyarova, K. Abdukhalikov, B. K. Sartayev, On the free metabelian Novikov and metabelian Lie-admissible algebras, Communications in Mathematics, 2025, 33(3 Special issue).

\bibitem{DauSar}
A. Dauletiyarova, B. K. Sartayev, Basis of the free noncommutative Novikov algebra, Journal of Algebra and its Applications, 2024, 2550292.

\bibitem{bicom3}
V. Drensky, N. Ismailov, M. Mustafa, B. Zhakhayev, Free bicommutative superalgebras, Journal of Algebra, 2024, 652, 158–187.

\bibitem{bicom4}
V. Drensky, B.K. Zhakhayev, Noetherianity and Specht problem for varieties of bicommutative algebras, Journal of Algebra, 2018, 499, 570–582.


\bibitem{Ism}
A.S. Dzhumadil’daev, N.A. Ismailov, Polynomial identities of bicommutative algebras, Lie and Jordan elements, Communications in Algebra, 2018, 46(12), 5241–5251.

\bibitem{DzhLofwall} 
A. S. Dzhumadil’daev, C. L\"ofwall, 
Trees, free right-symmetric algebras, free Novikov algebras and identities, 
Homology, Homotopy Appl., 4(2), 165--190 (2002).

\bibitem{assos1}
A.S. Dzhumadil'Daev, B. K. Zhakhayev, S. A. Abdykassymova, Assosymmetric operad, Communications in Mathematics , 2022, 30(1), 175–190.

\bibitem{bicom2}
N.A. Ismailov, F.A. Mashurov, B.K. Sartayev, On algebras embeddable into bicommutative algebras, Communications in Algebra, 2024, 52(11), 4778–4785.


\bibitem{KolMashSar}
P. Kolesnikov, F. Mashurov, B. Sartayev, On Pre-Novikov Algebras and Derived Zinbiel Variety, Symmetry, Integrability and Geometry: Methods and Applications (SIGMA), 2024, 20, 17.

\bibitem{KSO2019}
P. S. Kolesnikov, B. Sartayev, A. Orazgaliev,
Gelfand--Dorfman algebras, derived identities, 
and the Manin product of operads,
Journal of Algebra {539}, 260--284 (2019).

\bibitem{Binaryperm}
A. Kunanbayev, B. K. Sartayev, Binary perm algebras and alternative algebras, arXiv preprint arXiv:2309.09503, 2023.

\bibitem{assos2}
F. Mashurov, I. Kaygorodov, One-generated nilpotent assosymmetric algebras, Journal of Algebra and its Applications, 2022, 21(2), 2250031.


\bibitem{Meta-Mal}
S. V. Pchelintsev, Speciality of Metabelian Mal'tsev Algebras, Mathematical Notes, 74, 2003, p. 245-254.

\bibitem{erlagol2021}
B. Sartayev, P. Kolesnikov,
Noncommutative Novikov algebras, European Journal of Mathematics, 9(2), 35, (2023).

\end{thebibliography}
\end{document}